\documentclass{amsart}

\usepackage{natbib, color}
\usepackage{graphicx, float}
\usepackage{multicol}
\usepackage{amsmath, fullpage, booktabs, amsthm}
\usepackage[colorlinks=TRUE, citecolor=blue]{hyperref}
\usepackage{amsfonts, bm}
\usepackage[margin=1in]{geometry}
\usepackage{setspace}

\usepackage[tablesfirst, nolists]{endfloat}

\allowdisplaybreaks

\newtheorem{Proposition}{Proposition}
\newtheorem{Lemma}{Lemma}

\newcommand{\N}{{\mathcal{N}}}

\newcommand{\PP}{{\mathbb{P}}}
\newcommand{\EE}{{\mathbb{E}}}
\newcommand{\U}{{\mathbb{U}}}
\newcommand{\D}{{\mathbb{D}}}
\newcommand{\I}{{\mathbb{I}}}

\newcommand{\gv}{\,|\,}

\title{On a Variation of Gambler's Ruin Problem}

\author{Zhiyi Chi$^1$ and Vladimir Pozdnyakov$^1$$^2$}
\thanks{\\
1. Department of Statistics, University of Connecticut, 215 Glenbrook Road, Storrs, CT 06269-4120\\
2. Corresponding author: {\tt vladimir.pozdnyakov@uconn.edu}}

\date{}

\begin{document}

\maketitle

%\vspace{-1cm}
%\singlespacing
\doublespacing

%\begin{multicols}{2}

{
\noindent{\textbf{Abstract}}
Assume that letters (from a finite alphabet) in a text form a Markov chain. We track two distinct words, $U$ and $D$.
A gambler gains 1 point for each occurrence of $U$ (including overlapping occurrences) and loses 1 point for each
occurrence of $D$ (also including overlapping occurrences). We determine the probability of gaining $A$ points before
losing $B$ points, where $A$ and $B$ are integers. Additionally, we find the expected waiting time until one of the two
events---gaining $A$ points or losing $B$ points---occurs.


}

\section{Introduction}
Consider a text generated by a Markovian mechanism from a finite alphabet.
Fix two distinct words such that neither word is a subword of the other. A gambler gains 1 point for each
occurrence of the first word and loses 1 point for each occurrence of the second word.
Points are gained or lost for each overlapping occurrence. Let $A$ and $B$ be two integers.
The game stops if the gambler reaches $A$ points or loses $B$ points. We are interested in
the probability of reaching $A$ first and the expected duration of the game.

This is a variation of the classical ruin problem (for example, see \cite{Feller1968}).
A comprehensive review of the gambler's ruin problem (both first-step analysis and martingale
approaches), where the $\pm 1$ payments form a sequence of independent identically distributed
(i.i.d.) random variables, can be found in \cite{Steele2001}.
The case where the $\pm 1$ payments are Markovian was first addressed in \cite{Mohan1955}.

In this scenario, we have three different payments (1, 0, and -1).  The gambler's ruin problem
for this type of random walk with independent increments was examined in \cite{Gut2013}.
However, the sequence of payments from the two-word game does not form a Markov Chain,
even if the text is generated by an i.i.d. source. Nevertheless, we will demonstrate that,
through an appropriate embedding, this gambler's problem can be reduced to the gambler's ruin
problem for correlated random walks, as studied in \cite{Mohan1955}.

\section{A Gambler's Ruin Problem for Two-Word Game}

Let $\{Y_i\}_{i\geq 1}$ be a time-homogeneous irreducible Markov Chain with states
from a finite alphabet $\Delta$ with at least two letters. Let $$U=u_1\dots u_K \quad\mbox{ and }\quad
D=d_1\dots d_M$$ be two distinct {\it words} over alphabet $\Delta$, where $K,M\geq 1$. %and $\{u_k,d_m\}_{1\leq k\leq K,1\leq m\leq M}\in\Delta$.
We assume that that words $U$ and $D$ are not contiguous subsequences of each other and that both
\begin{equation}\label{word*probability}
\PP(Y_1=u_1,\dots,Y_K=u_K)>0\quad\quad
\mbox{and}\quad\quad
\PP(Y_1=d_1,\dots,Y_M=d_M)>0.
\end{equation}
{Since the finite Markov chain $\{Y_i\}_{i\geq 1}$ is irreducible, it follows from (\ref{word*probability})
that both words will occur infinitely often. Moreover, the expected waiting time until the first
occurrence of a word (or between consecutive occurrences) is finite.}

Let
$$
\U_i=\left\{\begin{array}{ll}
             0,& \mbox{ if } 1\leq i<K, \\
             \I_{Y_{i-K+1}=u_1,\dots,Y_i=u_K},& \mbox{ if } i\geq K.
           \end{array}
\right.
$$
and
$$
\D_i=\left\{\begin{array}{ll}
             0,& \mbox{ if } 1\leq i<M, \\
             \I_{Y_{i-M+1}=u_1,\dots,Y_i=u_M},& \mbox{ if } i\geq M.
           \end{array}
\right.
$$
Then the total number of points at time $n$ is given by  $$ S_n=\sum_{i=1}^n (\U_i-\D_i).$$
Given integers $A,B>0$, we define the following waiting time
\begin{equation*}\label{stopping*time}
\tau=\min\left\{ n\geq 1: S_n=A \mbox{ or } S_n=-B\right\}.
\end{equation*}
The objective is to find the probability $\alpha=\PP(S_\tau=A)$ and the
expected duration of the game $\EE(\tau)$.

However, note that $\PP(\tau < \infty)$ is not necessarily equal to 1. It is easy
to construct a Markov chain where the words $U$ and $D$ always appear in pairs.
For example, let $\{Y_i\}_{i\geq 1}$ be an i.i.d. sequence of letters from
$\Delta = \{0, 1\}$, with $\PP(1) = p$, $0 < p < 1$. Then, for the words $U = 10$ and $D = 01$,
$\PP(\tau = \infty)=1$ for any $A, B > 1$. Conditions on the transition matrix
of $\{Y_i\}_{i\geq 1}$ that guarantee $\PP(\tau < \infty) = 1$ for any two words $U$
and $D$ that satisfy (\ref{word*probability}) will be discussed in the last section.

\section{Two Embedded Markov Chains}\label{embedded*mc*section}

Define the stopping times
\begin{equation*}\label{first*word*stopping*time}
\tau_1=\min\{n\geq 1: \U_n=1 \mbox{ or } \D_n =1\},
\end{equation*}
and for $k>1$
\begin{equation*}\label{k*word*stopping*time}
\tau_k=\min\{n>\tau_{k-1}: \U_n=1 \mbox{ or } \D_n =1\}.
\end{equation*}
Let
$$
X_k=\left\{\begin{array}{rl}
           % \nonumber to remove numbering (before each equation)
             1,& \mbox{ if }  \U_{\tau_k}=1, \\
             -1,& \mbox{ if } \D_{\tau_k}=1.
           \end{array}
\right.
$$
Then $\{X_k\}_{k\geq 1}$ is a time-homogeneous two-state Markov chain with an initial
distribution and transition matrix that can be computed using standard methods.
More specifically, one can use techniques developed for the occurrence of patterns,
similar to how it is done in \cite{Pozdnyakov2025}, or apply the first-step analysis
to another Markov chain associated with the original Markov chain $\{Y_k\}_{k\geq 1}$.
In this paper, we will employ the first-step analysis, as described
in Section~\ref{standard*caculations}.

We also have that
$$S_{\tau_k}=X_1+\cdots+X_k,$$
and, as a consequence, $$\alpha=\PP(S_\tau=A)=\PP(X_1+\cdots+X_T=A),$$
where
$$T=\min\{k\geq 1: X_1+\cdots+X_k=A\mbox{ or }-B\}.$$ Thus, finding ruin
probability $\alpha$ and $\EE(T)$ is equivalent to the gambler's ruin problem
for correlated random walk (see \cite{Mohan1955}). {Note that $\PP(T < \infty) = 1$
if and only if $\PP(\tau < \infty) = 1$. A simple, word-specific necessary and
sufficient condition for $\PP(T < \infty) = 1$ is that the two-state
Markov chain $\{X_k\}_{k \geq 1}$ is aperiodic. All we need to do is exclude the
perfect alternation between $1$ and $-1$.}

However, since Markov chain $\{X_k\}_{k\geq 1}$ ignores zero payments, it cannot be
directly  used to determine the mean duration of the game, $\EE(\tau)$. For this, we need
to introduce a different embedded Markov chain. Consider the finite, {\it ordered in a particular way},
state space $\tilde{\Delta}$, which consists of the following words over the alphabet $\Delta$:
(1) all the letters from $\Delta$, (2) all the prefixes of the word $U$ (excluding the first letter,
but including $U$), and (3) all the prefixes of the word $D$ (excluding the first letter,
but including $D$). If two prefixes of $U$ and $D$ are identical, they count as one state.
Let
$$Z_k=\mbox{the longest suffix of word $Y_1Y_2\dots Y_k$ that coinsides with a word from $\tilde{\Delta}$}.$$
Then $\{Z_k\}_{k\geq 1}$ is a time-homogeneous finite-state Markov chain.
A similar look-back construction in the context of pattern occurrence was first introduced in \cite{GerberLi1981}.
{Next, let us give simple but important statements about Markov chain $\{Z_k\}_{k\geq 1}$.
\begin{Proposition}\label{Markov*Chain*Z}
If words $U$ and $D$ satisfy (\ref{word*probability}), then
\begin{enumerate}
\item Markov chain $\{Z_k\}_{k\geq 1}$ has exactly one recurrent positive class,
\item both $U$ and $D$  belong to that class,
\item Markov chain $\{Z_k\}_{k\geq 1}$ has a unique stationary distribution.
\end{enumerate}
\end{Proposition}
The proof is straightforward, and it is based on the following key observations:
(1) states $U$ and $D$ are recurrent and positive and
(2) both states $U$ and $D$ are reachable from  any other states in $\tilde{\Delta}$.}

Note that $$\I_{Z_{\tau_k}=U}=\I_{X_k=1}
\mbox{ and }\I_{Z_{\tau_k}=D}=\I_{X_k=-1},$$ and, as a consequence, both the sequence of stopping
times $\{\tau_k\}_{k\geq 1}$ and Markov chain $\{X_k\}_{k\geq 1}$ can be introduced via embedded
Markov chain $\{Z_k\}_{k\geq 1}$. More specifically,
\begin{equation*}\label{first*word*stopping*time*via*Y}
\tau_1=\min\{n\geq 1: Z_n=U \mbox{ or } Z_n =D\},
\end{equation*}
for $k>1$,
\begin{equation*}\label{k*word*stopping*time*via*Y}
\tau_k=\min\{n>\tau_{k-1}: Z_n=U \mbox{ or } Z_n =D\},
\end{equation*}
and
\begin{equation*}\label{X*via*Y}
X_k=\left\{\begin{array}{rl}
           % \nonumber to remove numbering (before each equation)
             1,& \mbox{ if }  Z_{\tau_k}=U, \\
             -1,& \mbox{ if } Z_{\tau_k}=D.
           \end{array}
\right.
\end{equation*}

Let $\gamma_1=\tau_1$ and $\gamma_k=\tau_k-\tau_{k-1}$, $k\geq 2$, then
$$\tau=\tau_T=\gamma_1+\cdots+\gamma_T.$$
Our next step is to relate $\EE(\tau)$ and $\EE(T)$.

\section{Formula for $\EE(\tau)$ via $\alpha$ and $\EE(T)$}

By the strong Markov property, conditional on $X_1$, $\gamma_1$ is
independent of the other $X_j$'s, and conditional on $X_{i-1}$ and
$X_i$, $i\ge2$, $\gamma_i$ is independent of the other $X_j$'s.  Then we get
\[
  \EE(\gamma_j\gv X_1,\dots, X_{T})
  =
  \begin{cases}
    \EE(\gamma_1\gv X_1), & \text{if~} j=1,\\
    \EE(\gamma_j\gv X_{j-1}, X_j), &
    \text{else}.
  \end{cases}
\]
Let
\[
  \Gamma_n = \EE(\gamma_1\gv X_1) + \sum_{2\le j\le n} \EE(\gamma_j \gv
  X_{j-1}, X_j).
\]
It follows that
\[
  \EE[\tau]
  =
  \EE[\EE(\gamma_1+\cdots+ \gamma_{T}\gv X_1,\dots, X_{T})]
  = \EE(\Gamma_{T}).
\]
Since the $X_j$'s are Markov,
\begin{align*}
  \xi_1& = \EE(\gamma_1\gv X_1) - \EE\big[\EE(\gamma_1\gv X_1)\big]
  = \EE(\gamma_1\gv X_1) - \EE(\gamma_1),
  \\
  \xi_j
  &= \EE(\gamma_j\gv X_{j-1}, X_j)
  - \EE\big[\EE(\gamma_j\gv X_{j-1}, X_j)\gv X_{j-1}\big]
  = \EE(\gamma_j\gv X_{j-1}, X_j)
  - \EE(\gamma_j\gv X_{j-1}),
\end{align*}
are martingale differences with respect to the filtration generated by
the $X_j$'s.  Then by Optional Stopping Theorem we have that
\[
  \EE(\Gamma_{T})
  =
  \EE\left[\sum^{T}_{j=1} \xi_j
    +
    \EE(\gamma_1) + \sum^{T}_{j=2}
    \EE(\gamma_j\gv X_{j-1})\right]
  =
  \EE(\gamma_1) + \EE\left[\sum^{T-1}_{j=1}
    \EE(\gamma_{j+1}\gv X_j)\right].
\]
One can verify that
\[
  \EE(\gamma_{j+1}\gv X_j)
  = a X_j + b,
\]
where
\[
  a=\frac{1}{2}\big[\EE(\gamma_2\gv X_1=1) - \EE(\gamma_2\gv X_1=-1)\big],
  \quad
  b=\frac{1}{2}\big[\EE(\gamma_2\gv X_1=1) + \EE(\gamma_2\gv X_1=-1)\big].
\]
Then
\begin{align*}
  \EE\left[\sum^{T-1}_{j=1}
    \EE(\gamma_{j+1}\gv X_j)\right]
  &=
  \EE\left[a \sum^{T-1}_{j=1} X_j+ b(T-1)\right]
  %\\&
  =
  a [\EE(S_{\tau}) - \EE(X_{T})]
  + b[\EE(T) - 1].
\end{align*}
Since $$\EE(S_{\tau}) = A \alpha - B(1-\alpha)=(A+B)\alpha-B$$
and
%\[
%  X_{T}=
%  \begin{cases}
%    1, &\text{if~}X_T=A \\
%    -1,&\text{if~}X_T=-B,
%  \end{cases}
%\]
$$\EE(X_{T}) =\EE(\U_\tau-\D_\tau)= 2\alpha-1,$$
by combining the above formulas, we get the following result. {
\begin{Proposition}
If Markov chain  $\{X_k\}_{k\geq 1}$ ia aperiodic, then
\begin{equation}\label{tau*via*T}
  \EE(\tau)
  =
  \EE(\gamma_1) + a\big[(A + B-2)\alpha - (B-1)\big]
  + b\big[\EE(T) - 1\big].
\end{equation}
\end{Proposition}}

\section{First-Step Analysis}\label{standard*caculations}

Let $\pi(z)=\PP(Z_1=z)$, $z \in \tilde{\Delta}$, be the initial distribution of Markov chain $\{Z_k\}_{k\geq 1}$.
Note that only one-letter words from $\tilde{\Delta}$ have non-zero probabilities and these are equal to the
initial probabilities of the corresponding letters in the original Markov Chain $\{Y_k\}_{k\geq 1}$.
Let $\mathbf{P}=(P_{st})_{s,t\in\tilde{\Delta}}$ be the transition matrix of $\{Z_k\}_{k\geq 1}$.

First, we will derive the initial distribution and transitional probabilities of the two-state Markov
chain  $\{X_k\}_{k\geq 1}$.  Define the hitting time
\[
 \N = \min\{n>1: Z_n\in\{U,D\}\}.
\]
Then
\[
  \PP(X_1=1)
  =
  \pi(U) + \sum_{z\notin \{U,D\}} \PP(Z_1=z) \PP(Z_\N=U\gv Z_1=z),
\]
and
\[
  \PP(X_1=-1)
  =
  \pi(D) + \sum_{z\notin \{U,D\}} \PP(Z_1=z)\PP(Z_\N=D\gv Z_1=z).
\]
On the other hand, by strong Markov property,
\[
  \PP(X_2=1\gv X_1=1) =
  1 - \PP(X_2=-1\gv X_1=1)=\PP(Z_\N=U\gv Z_1=U)
\]
and
\[
  \PP(X_2=1\gv X_1=-1)=
  1 - \PP(X_2=-1\gv X_1=-1)=\PP(Z_\N=U\gv Z_1=D)
\]
So, it boils down to finding $p_z=\PP(Z_\N = U \gv Z_1=z)$ for every state $z\in \tilde{\Delta}$. For each $z$, by the first-step
analysis, we  have
\[
  p_z = P_{zU} + \sum_{s\notin \{U,D\}} P_{zs} p_s.
\]
Let $\mathbf{p}$ be the column-vector of $p_z$,  $\mathbf{p}_U$ be the column-vector of $P_{zU}$, $\mathbf{I}$ be the identity matrix, and $\mathbf{Q}$ be the matrix with
$$Q_{st} = P_{st} \I_{t\ne \{U,D\}}.$$
Then $$\mathbf{p} = \mathbf{p}_U + \mathbf{Q} \mathbf{p}.$$ Additionally, the following is true.
\begin{Lemma} Matrix $\mathbf{I}-\mathbf{Q}$ is invertible.
\end{Lemma}
\begin{proof}
Since all the entries of $\mathbf{Q}$ are nonnegative, by Perron--Frobenius theorem,
its spectral radius $\rho$ is also an eigenvalue, and there is a corresponding
left eigenvector $\mathbf{q}$ whose entries are nonnegative with a sum equal to 1.
If $\rho<1$, then $\mathbf{I}-\mathbf{Q}$ a non-singular $M$-matrix.
Since $\mathbf{P}\ge \mathbf{Q}$ (i.e., all entries of $\mathbf{P}-\mathbf{Q}$ are
nonnegative), then $\rho$ cannot be greater than the spectral radius of $\mathbf{P}$, which is 1.

Assume that  $\rho=1$.  Then the eigenvector $\mathbf{q}$ satisfies
\[
    \mathbf{q}' = \mathbf{q}'\mathbf{Q}
\]
with $q_U=0$ and $q_D=0$.
Now, note that entry-wise
\[
    \mathbf{q}' = \mathbf{q}'\mathbf{Q}\leq \mathbf{q}'\mathbf{P}.
\]
Since the components of both the vector $\mathbf{q}'$ and the vector $\mathbf{q}' \mathbf{P}$
are non-negative and sum to 1, we, in fact, have
\[
    \mathbf{q}' = \mathbf{q}'\mathbf{P}.
\]
That is, $\mathbf{q}$ is a stationary distribution of Markov chain $\{Z_k\}_{k\geq 1}$.
But according to Proposition~\ref{Markov*Chain*Z}, the stationary distribution
is unique, $U$ and $D$ are recurrent positive states. Therefore, $q_U$ and $q_D$ must be
strictly positive, which leads to a contradiction.
\end{proof}
Since $\mathbf{I}
- \mathbf{Q}$ is invertible, we get that
\begin{equation}\label{p*for*X}
\mathbf{p} = (\mathbf{I} - \mathbf{Q})^{-1} \mathbf{p}_U.
\end{equation}
Second, let us calculate $\EE(\gamma_1)$, $\EE(\gamma_2\gv X_{1}=1)$, and $\EE(\gamma_2\gv X_{1}=-1)$.
Note that
\[
  \EE(\gamma_1) = \pi(U)+\pi(D)+ \sum_{s\notin \{U,D\}} \pi(s)\EE(\N\gv Z_1=s).
\]
Also, $\EE(\gamma_2 \gv X_{1}=1) = \EE(\N\gv Z_1=U)-1$ and $\EE(\gamma_2 \gv X_{1}=-1) = \EE(\N\gv Z_1=D)-1$.
Thus, all we need to do is calculate $e_z=\EE(\N\gv Z_1=z)$ for every state $z$. Then, by applying the first-step
analysis we  have
\[
  e_z = 2P_{zU}+2P_{zD}+ \sum_{s\notin \{U,D\}} P_{zs}\big[\EE(\N\gv Z_1=s)+1\big]=1+P_{zU}+P_{zD}+ \sum_{s\notin \{U,D\}} P_{zs}\EE(\N\gv Z_1=s).
\]
Let $\mathbf{e}$ be the column-vector of $e_z$, $\mathbf{p}_D$ be
the column-vector of $P_{zD}$, and $\mathbf{1}$ be the column-vector of ones.  Then
\begin{equation}\label{e*for*gamma}
\mathbf{e} = (\mathbf{I} - \mathbf{Q})^{-1} [\mathbf{1}+\mathbf{p}_U+\mathbf{p}_D].
\end{equation}

\section{An Example}

Consider a sequence generated by flips of a fair coin. That is, $\{Y_k\}_{k\geq 1}$ is a
sequence of independent identically distributed letters over alphabet $\Delta=\{0,1\}$
with $\PP(0)=\PP(1)=1/2$. Let the word $U=11$ and the word $D=01$. Then the state space of
the embedded  Markov chain  $\{Z_k\}_{k\geq 1}$ is $\tilde{\Delta}=\{1,0,11,01\}$. The
initial distribution $(1/2,1/2,0,0)$ and the transition matrix
$$
\mathbf{P}=\begin{pmatrix}
               0 & 1/2 & 1/2 & 0 \\
               0 & 1/2 & 0 & 1/2 \\
               0 & 1/2 & 1/2 & 0 \\
               0 & 1/2 & 1/2 & 0 \\
             \end{pmatrix}.
$$
Then
$$
\mathbf{p}_U=\begin{pmatrix}
               1/2  \\
               0  \\
               1/2  \\
               1/2  \\
             \end{pmatrix},\quad\quad
\mathbf{p}_D=\begin{pmatrix}
               0  \\
               1/2  \\
               0  \\
               0  \\
             \end{pmatrix},\quad\quad
\mathbf{Q}=\begin{pmatrix}
               0 & 1/2 & 0 & 0 \\
               0 & 1/2 & 0 & 0 \\
               0 & 1/2 & 0 & 0 \\
               0 & 1/2 & 0 & 0 \\
             \end{pmatrix}.
$$
Formula (\ref{p*for*X}) gives us
$$
\mathbf{p}'=\begin{pmatrix}
               1/2  &
               0    &
               1/2  &
               1/2
             \end{pmatrix}.
$$
As a consequence, we have
$$\PP(X_1=1)=1/4,\quad\quad\PP(X_2=1\gv X_1=1)=1/2,\quad\quad \PP(X_2=1\gv X_1=-1)=1/2.$$
Using results on the gambler's ruin problems for correlated random walk we obtain that
$$\alpha=\frac{B-1/2}{A+B},$$
and
$$\EE(T)=AB+(B-A)/2.$$
Next, formula (\ref{e*for*gamma}) leads us to
$$
\mathbf{e}'=\begin{pmatrix}
               3 &
               3 &
               3 &
               3
             \end{pmatrix},
$$
and, therefore,
$$\EE(\gamma_1=1)=3,\quad\quad\EE(\gamma_2\gv X_1=1)=2,\quad\quad \EE(\gamma_2\gv X_1=-1)=2.$$
Finally, using (\ref{tau*via*T}) we get that
$$\EE(\tau)=2AB+B-A+1.$$

It is known (for instance, see \cite{Lladser2007}) that the construction of  Markov chain $\{Z_k\}_{k\geq 1}$ given in
Section~\ref{embedded*mc*section} is not always optimal. Note that in our example of
the $11$ vs $01$ game the state $1$ is a transient state of $\{Z_k\}_{k\geq 1}$. This means that a
Markov chain with a smaller state space can be embedded. This was exploited in
\cite{Pozdnyakov2025}, where both $\alpha$ and $\EE(\tau)$ were found by building
appropriate martingales for correlated random walks with delays. However, the martingale
approach of \cite{Pozdnyakov2025} cannot be used if we consider a game with different
words (for example, $11$ vs $00$).

The progress  made here is twofold. First, the method will work even if
the text is not generated by an i.i.d. source. Second, and perhaps more importantly,
the expected waiting time $\EE(\tau)$ cannot be computed with help of Markov chain
$\{X_k\}_{k\geq 1}$ alone, because it ignores zero payments. For this, we need
Markov chain $\{Z_k\}_{k\geq 1}$ and formula~(\ref{tau*via*T}) that connects
$\EE(\tau)$  and $\EE(T)$.

\section{When is $\tau$ finite?}

As we mentioned above, if two words $U$ and $D$ are given to us, then in order to
check that $\PP(\tau<\infty)=1$, we need to construct the two-state Markov chain
$\{X_k\}_{k\geq 1}$ and verify that it is aperiodic. But another interesting question
is: what conditions on the Markov chain of letters $\{Y_k\}_{k\geq 1}$ are needed to guarantee
that for any two words $U$ and $D$ that satisfy (\ref{word*probability})
we have $\PP(\tau<\infty)=1$?

The answer turns out to be not that trivial. First of all, it depends on the size of
the alphabet $\Delta$. If $|\Delta|=2$, there are pairs of words $U$ and $D$ for
which the two-state Markov chain $\{X_k\}_{k\geq 1}$ is periodic for almost any transition
matrix of the Markov chain $\{Y_k\}_{k\geq 1}$. More specifically, we have the following result.
\begin{Proposition}
Let the Markov chain $\{Y_k\}_{k\geq 1}$ with state set $\Delta=\{0,1\}$ be irreducible and
aperiodic (or, equivalently, its transition matrix is strictly positive).
Then the Markov chain $\{X_k\}_{k\geq 1}$ is periodic if and only if
we have one of the following combinations of words: $U = y^t(1-y)^s$ and
$D = (1-y)^sy^t$, where $y \in \Delta$, $t, s \geq 1$, but either $s = 1$ or
$t = 1$. Here, $y^t$ denotes a word formed by repeating the letter $y$ $t$ times.
\end{Proposition}
\begin{proof}
Assume that $U$ is always followed by $D$, and $D$ is always followed by $U$.
First, it is obvious that neither $U$ nor $D$ are runs; both must include both
the letters $0$ and $1$. Assume the last letter in $D$ is $1$. Consider the following
possible text: $U0^nU$, where $n$ is greater than the maximum of  lengths of $D$ and $U$.
Since $U$ can only occur at the beginning and the end of the text $U0^nU$,
$D$ must occur in the middle. Moreover, the last letter of $D$, which is $1$,
must be in the second $U$ at the end.  This means that $D$ starts with a run of $0$
(with a length of at least one).

Next, let us consider another possible text: $U0^n1^nU$. If $U=0^i1^j$ (that is, there
is a $U$ in the middle of $U0^n1^nU$), then by considering $U0^nU$, we get that $D=0^t1^s$,
where $t, s \geq 1$. If $U\neq 0^i1^j$, then there are only two $U$s in $U0^n1^nU$, and the
last $1$ of $D$ must be in $1^nU$. Therefore, again $D=0^t1^s$, where $t, s \geq 1$.
Using the same argument, we get that $U=0^i1^j$ or $U=1^i0^j$, where $i, j \geq 1$.
But by considering the text $DD$, we see that $U=0^i1^j$ must be excluded because
it cannot occur in $DD$ (remember that $U$ is not a subword of $D$). Thus, $U=1^i0^j$. If $t<i$,
then $U$ cannot occur in $DD$. Therefore, $t \geq i$.
By considering the occurrence of $D$ in $UU$, we get that $t=i$.
Using similar arguments, we also have $s=j$, that is, $D=0^s1^t$ and $U=1^t0^s$.

Finally, if both $s,t>1$, consider the text $D(01)^nD$ and observe that $U$ cannot
occur in this text. Therefore, either $s=1$ or $t=1$ (or both). A direct check
shows that if $D=0^s1^t$ and $U=1^t0^s$, with $t,s\geq 1$, either $s = 1$ or
$t = 1$, then $U$ and $D$ alternates in any text generated by $\{Y_k\}_{k\geq 1}$.
The other two cases are obtained by interchanging $0$ and $1$.
\end{proof}

The case $|\Delta|>2$ is different. For example, the following is true.
\begin{Proposition}
Suppose that Markov chain $\{Y_k\}_{k\geq 1}$ with finite state set $\Delta=\{0,1,2,\dots\}$
has a strictly positive transition matrix. Then the Markov chain $\{X_k\}_{k\geq 1}$ is
aperiodic  for any two words $U$ and $D$.
\end{Proposition}
\begin{proof}
Assume that the words $D$ and $U$ alternate. As before, neither $U$ nor $D$ can be runs.
Consider the text $U0^nU$, where $n$ is greater than the length of $D$. The occurrence
of $D$ along this path implies that $D$ either starts with $0$ or ends with $0$. Now, if
we consider the text $U1^nU$, it indicates that $D$ either starts with $1$ or
ends with $1$. Therefore, $D$ must be of the form $0\dots 1$ or $1\dots 0$.
However, this contradicts the requirement that $D$ must also occur in the text $U2^nU$.
\end{proof}
However, the exact characterization of  Markov chains $\{Y_k\}_{k\geq 1}$ that guarantees
$\PP(\tau<\infty)=1$ for any two words $U$ and $D$ that satisfy (\ref{word*probability})
is an open question.

\section{Concluding Remarks}

Note that the developed technique allows us to find the ruin probability and the mean duration time
for the gambler’s ruin problem in a special case of the so-called {\it Markov random walk}
(see, for example, \cite{Grama_etal2018}). In our case, the walk goes up by 1 when
the associated finite Markov chain is in a particular state, goes down by 1 when it is in another
state, and remains unchanged when the chain is in other states.

\bibliographystyle{mcap}
\bibliography{crwm}

\end{document}